\documentclass[a4paper,12pt]{article}
\usepackage[utf8]{inputenc}
\usepackage[T1]{fontenc}
\usepackage{lmodern}
\usepackage{amsmath,amssymb,amsthm}
\usepackage[final,dvips,breaklinks]{hyperref}
\usepackage[dvips,final]{graphicx}		
\usepackage{psfrag}
\usepackage{cancel}

\newcommand{\du}{\mathrm{d}u}
\newcommand{\dx}{\mathrm{d}x}
\newcommand{\dy}{\mathrm{d}y}
\newcommand{\R}{\mathbb{R}}
\newcommand{\Z}{\mathbb{Z}}
\newcommand{\set}[2]{\left\{\left. #1 \, \right| \, #2\right\}}
\DeclareMathOperator{\id}{Id}

\newtheorem{cor}{Corollary}[section]
\newtheorem{theo}[cor]{Theorem}
\newtheorem{lem}[cor]{Lemma}

\theoremstyle{definition}
\newtheorem{df}[cor]{Definition}
\newtheorem{nt}[cor]{Notation}

\newtheoremstyle{example}
	{}{}
	{}
	{}
	{\bfseries}
	{.}
	{\newline}
	{}

\theoremstyle{example}
\newtheorem{ex}{Example}[section]

\title{Ideas for best teaching integrals: we are teaching wrongly and how to do it right}
\author{Bruno Mendonça Rey dos Santos}
\begin{document}

\maketitle

\begin{abstract}
	We discuss some problems with the indefinite integral notation and the way of teaching of integrals in Calculus. Based on the discussion, and in order to avoid mistakes, we propose another notation for indefinite integrals.
\end{abstract}

\section*{Introduction}

When we teach indefinite integrals, we use the notation $\int$. But it is not clear for the students what the $\int$ means. What does it means $\int f(x) \dx$? Can one write $F(x) = \int f(x) \dx$? Is $\int f(x) \dx$ a new function, or it is family of functions? Or does it depend on the context?

We use to write $\int x \dx = \frac{x^2}{2}$, but also $\int x \dx = \frac{x^2}{2} + c$, where $c \in \R$ is any constant. But our students don not know which one to use and why. So our notation is confusing.

Rather than that, we use this notations to teach our students some "mathemagical" manipulation of $\dx$, $\du$ and $\dy$. Those manipulations are useful but we do not prove them and they can lead us (or at least our students) to many mistakes. In the following lines, we will discuss some mistakes made by our students and mistakes we teach them how to make. We will also propose a (not so) new way of teaching Calculus using another notation.

This paper have two sections. In the first one we show examples of how we teach and we discuss what is wrong with each example. In the second section we propose a new notation for the indefinite integral and we solve each example already discussed using the new notation.

After I had prepared this material, I learned from one of my Calculus students that a similar notation had already appeared in a series of MIT video lectures available through the Internet: the \textit{Calculus Revisited: Single Variable Calculus}, whose instructor was Prof. Herbert Gross. For more information about this lectures see \cite{CalcRev}.

\tableofcontents

\section{How we use to teach}

In this section we show some examples of how we teach integrals to our students and we discuss what is wrong with each the examples.

\subsection{The integral $\int\frac{1}{x}\dx$}
	
	We teach our students that $\int \frac{1}{x} \dx = \ln|x| + c$, where $c \in \R$ is any constant. So lets consider the function $f \colon \R \to \R$ given by
	\[f(x) =
	\begin{cases}
		\ln(x) + 2, & \text{if} \ x > 0; \\
		\ln(-x) - \pi, & \text{if} \ x < 0.
	\end{cases}\]
	Clearly, $f'(x) = \frac{1}{x}$ and $f(x) \ne \ln|x| + c$.

	This example shows that the notation $\int$ does not take into consideration the function's domain, just its rule. When we take the function $x \mapsto \frac{1}{x}$ and do not say nothing about its domain, we are considering that the function's domain is the biggest set where the rule makes sense. In the Calculus context, it is the set $\R^* = \R\setminus \{0\}$.

	The notation $\int$ is misleading because it does not take into consideration the function's domain. But, in this particular case, the biggest problem is not the notation: we should take the function's domain into account and teach our students to do the same.

\subsection{Changing variables and some "mathemagic"}\label{magic}

	Here we give some examples of the practical way we teach our students to calculate some integrals changing the variables. Then we discuss the problems of this approach.
	
	\begin{ex}[$\int\cos \left(x^2\right) x \dx$]
		We use the following change of variables:
		\[u = x^2 \Rightarrow \frac{\du}{\dx} = 2x \Rightarrow x\dx = \frac{1}{2} \du.\]
		Then,
		\[\int \cos \left(x^2\right) x \dx = \int \cos(u)\cdot \frac{1}{2} \du = \frac{\sin u}{2} + k = \frac{\sin\left(x^2\right)}{2} + k.\]
	\end{ex}

	\begin{ex}[$\int \frac{1}{\sqrt{x^2 +1}} \dx$.]
		Lets analyze the following triangle:
		\begin{center}
			\psfrag{1}{$1$}
			\psfrag{x}{$x$}
			\psfrag{b}{$\sqrt{x^2+1}$}
			\psfrag{c}{$\theta$}
			\includegraphics{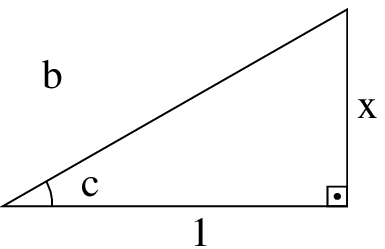}
		\end{center}
		 
		Based on the triangle above, if we call $x = \tan \theta$, we have that
		\[\frac{1}{\sqrt{x^2+1}} = \cos \theta \quad \text{and} \quad \frac{\dx}{\mathrm{d}\theta} = \sec^2\theta \Rightarrow \dx = \sec^2\theta \mathrm{d}\theta.\]
		Then,
		\[\int \frac{1}{\sqrt{x^2 +1}} \dx = \int \cos\theta \sec^2 \theta \mathrm{d}\theta = \int \sec\theta \mathrm{d}\theta.\]
		At this point, if we already know $\int \sec \theta \mathrm{d}\theta$, we can write
		\[\displaystyle \int \frac{1}{\sqrt{x^2 +1}} \dx = \ln\left| \tan \theta + \sec \theta \right| + k = \ln\left|x + \sqrt{1+x^2} \right| + k.\]
	\end{ex}

	\begin{ex}[$\int \sqrt{1-x^2} \dx$]\label{ex nao difeo}
		Lets now use the following triangle:
		\begin{center}
			\psfrag{1}{$1$}
			\psfrag{x}{$x$}
			\psfrag{b}{$\sqrt{1-x^2}$}
			\psfrag{c}{$\theta$}
			\includegraphics{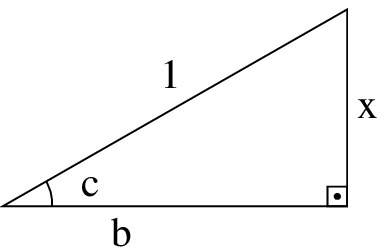}
		\end{center}
		Calling $x= \sin \theta$, we have that
		\[\sqrt{1-x^2} = \cos \theta \quad \text{and} \quad \frac{\dx}{\mathrm{d}\theta} = \cos \theta \Rightarrow \dx = \cos\theta \mathrm{d}\theta.\] Thus
		\begin{align*}
			& \int \sqrt{1-x^2} \dx = \int \cos(\theta) \cdot\cos(\theta) \mathrm{d}\theta = \int \cos^2(\theta) \mathrm{d}\theta = \\
			& = \frac{\theta}{2} + \frac{\sin(\theta)\cos(\theta)}{2} + k = \frac{\arcsin(x)}{2} + \frac{x\sqrt{1-x^2}}{2} + k.
		\end{align*}
	\end{ex}

	\begin{ex}[$\int_0^{\frac{3\pi}{4}} e^{\cos(x)}\cdot \sin(x) dx$]
		Using the change of variables $u = \cos(x)$, we have that
		\[\begin{cases}
			x = 0 \Rightarrow u = 1; \\
			x = \frac{3 \pi}{4} \Rightarrow u = -\frac{\sqrt 2}{2}; \\
			\du = -\sin(x)\dx;
		\end{cases}\]
		thus
		\[\int_0^{\frac{3\pi}{4}} e^{\cos(x)}\cdot \sin(x) \dx = \int_1^{-\frac{\sqrt 2}{2}} -e^u \du = \int_{- \frac{\sqrt 2}{2}}^1 e^u \du = \left.e^u\right|_{-\frac{\sqrt 2}{2}}^1 = e-\frac{1}{e^{\frac{\sqrt 2}{2}}}.\]
	\end{ex}

	Lets point some observations about our method of changing variable at calculating integrals:
	\begin{enumerate}
		\item The very good thing about this method is that it is fast and we can calculate integrals using just a feel lines.
		
		\item One problem about this method is that we make lots of calculations multiplying and dividing by $\dx$, $\du$ and $\mathrm{d}\theta$. But $\dx$, $\du$ and $\mathrm{d}\theta$ are (today) just symbols and we cannot add, subtract, multiply or divide by them.
		
		\item In the change of variables in the examples above, we have used four different theorems, one for each example, and each theorem has different hypothesis. Using the traditional method, we do not check the theorems hypothesis before making calculations. Our students do not even know the theorems. So, we are tanking conclusions (calculations) based on theorems whose hypothesis we have not checked. It is a dangerous thing to take conclusions without checking the theorems hypothesis, and it is even more dangerous to teach students to do the same.
	\end{enumerate}
	

%

%

\subsection{The integral $\int\frac{1}{1-\cos x + \sin x}\dx$} \label{last}
	
	The ideas to calculate the integral $\int\frac{1}{1-\cos x + \sin x}\dx$ were found in \cite{Gui}.
	
	Using the identities
	\[\cos x = \frac{1-\tan^2\frac{x}{2}}{1+\tan^2\frac{x}{2}} \quad \text{and} \quad \sin x = \frac{2\tan\frac{x}{2}}{1+\tan^2\frac{x}{2}},\]
	we can write
	\[\int \frac{1}{1-\cos x + \sin x}\dx = \int \frac{1+\tan^2\frac{x}{2}}{2\tan^2\frac{x}{2}+2\tan\frac{x}{2}}\dx.\]

	Now we can use the following change of variables:
	\[u=\tan\frac{x}{2} \Rightarrow \frac{\du}{\dx} = \frac{1}{2}\left(1+ \tan^2\frac{x}{2}\right) \Rightarrow \dx = \frac{2}{1+u^2}\du.\]
	Thus,
	\begin{align*}
		& \int \frac{1}{1-\cos x + \sin x}\dx =  \int \frac{1+u^2}{2u^2+2u}\cdot \frac{2}{1+u^2}\du = \int \frac{1}{u(u+1)}\du = \\
		&= \int \frac{1}{u} - \frac{1}{u+1}\du = \ln|u| -\ln|u+1| + k = \ln\left|\frac{u}{u+1}\right|+ k = \\
		&= \ln\left| \frac{\tan\frac{x}{2}}{1+\tan\frac{x}{2}}\right| + k.
	\end{align*}

	But,
	\begin{align*}
		& \frac{\tan\frac{x}{2}}{1+\tan\frac{x}{2}} = \frac{\sin\frac{x}{2}}{\cos\frac{x}{2}+\sin\frac{x}{2}} = \frac{\sin\frac{x}{2}\cos\frac{x}{2}}{\cos^2\frac{x}{2} + \sin\frac{x}{2}\cos\frac{x}{2}} = \\
		& = \frac{\frac{\sin x}{2}}{\frac{1+\cos x}{2}+\frac{\sin x}{2}} = \frac{\sin x}{1 + \cos x + \sin x}.
	\end{align*}
	Therefore,
	\[\int \frac{1}{1-\cos x + \sin x}\dx = \ln \left|\frac{\sin x}{1 + \cos x + \sin x}\right| + k.\]

	All these calculations are not entirely wrong, they are necessary. But lets take a look at both functions we have:
	\[x \mapsto \frac{1}{1-\cos x + \sin x} \quad \text{and} \quad x \mapsto \ln\left|\frac{\sin x}{1 + \cos x + \sin x}\right| + k.\]
	Its easy to see that the first function is defined in $x = \pi$, while the last one is not. That means that the domain of our solution is not the same domain of our original function, which is absurd.

	Using the same technique to calculate $\int \sqrt{1-\cos(x)} \dx$, we also find a primitive whose domain is smaller than the domain of the original function.

\subsection{The fundamental theorems of Calculus}

	The following two Theorems below be found in \cite{Spi}.

	\begin{theo}[First Fundamental Theorem of Calculus]\label{F1}
		Let $I$ be an interval, $x_0 \in I$ and $f \colon I \to \R$ a function. If $f$ is \textbf{continuous}, then function $F \colon I \to \R$ given by $F(x) = \int_{x_0}^x f(t)dt$ is a $C^1$ function and $F'(x) = f(x)$, for all $x \in I$.
	\end{theo}

	\begin{theo}[Second Fundamental Theorem of Calculus]\label{F2}
		Let $f \colon [a,b] \to \R$ a function. If $f$ is \textbf{integrable} and $F \colon [a,b] \to \R$ is a function such that $F'(x) = f(x)$, for all $x \in [a,b]$, then $\int_a^b f(x)\dx = F(b) - F(a)$.
	\end{theo}

	If we want to calculate $\int_a^b f(x) \dx$, we first calculate $F(x) = \int f(x) \dx$. Then we can write $\int_a^b f(x) \dx = F(b) - F(a)$. The problem here is that we use the same notation to calculate $F$ and to calculate $\int_a^b f(x)\dx$. But they are completely different problems. Besides that, once we get used to this calculations, we forget that $f$ must be integrable and we start to think that if $F(x) = \int f(x)\dx$, then $f$ is integrable and $\int_a^b f(x) \dx = F(b) - F(a)$.

	In \cite{GO} we can find functions $F$ and $f$, defined in the same closed interval, and such that $F = \int f(x) \dx$ but $f$ is not integrable. See also \cite{Volt}.

\section{How we should teach}

The main idea to make things right is to teach our students to make calculations the same way we prove the theorems. The first part in order to prove theorems in Calculus, is to have good definitions

\subsection{Definitions}

	We will first use a good, but not so formal, definition of function:

	\begin{df}[Function]
		A function is an object formed by 3 parts:
		\begin{enumerate}
			\item a set $A$ called the \textbf{domain} of the function,
			\item a set $B$ called the \textbf{codomain} of the function and
			\item a rule that relates each element $x \in A$ to a unique element $y \in B$.
		\end{enumerate}
	
		We use the notation $f(x)$ to denote the only element $y \in B$ which is related to the element $x \in A$ by the rule of the function $f$, ie., $f(x)=y$.
	\end{df}

	The 3 parts together ($A$, $B$ and the rule $x \mapsto f(x)$) is called function. If we want to give a name for a function, generally we use roman letters, for example, lets consider the function whose domain is $A = [0,\infty[$, codomain is $B=\R$ and its rule relates each number $x \in [0,\infty[$ to is square root $\sqrt{x}$. We can call this function $f$. To define this function $f$, we could just write
	\[\begin{matrix}
	f : & [0,\infty[ & \longrightarrow & \R \\
	& x & \longmapsto & \sqrt{x}
	\end{matrix}\]

	If we write
	\[\begin{matrix}
		g : & A & \longrightarrow & B \\
		& x & \longmapsto & g(x)
	\end{matrix}\]
	we mean that $g$ is the name of a function, $A$ is the domain of $g$, $B$ is the codomain of $g$ and $x \mapsto g(x)$ is the rule of the function $g$.

	In Calculus, it is common to give just the rule of some function, and the function's domain and codomain are not explicitly given. For example, lets consider the function $\frac{1}{x}$. Here we mean the function whose rule is $x \mapsto \frac{1}{x}$. We have not said a word about the function's domain nor its codomain, but we are considering that its domain is the biggest subset of $\R$ where the rule makes sense and that the codomain is $\R$ itself. In other words, when we say "lets consider the function $\frac{1}{x}$", we mean the following function:
	\[\begin{matrix}
		\R^* & \longrightarrow & \R \\
		x & \longmapsto & \frac{1}{x}
	\end{matrix}\]

	\begin{nt}
		If $g$ is a function, $D_g$ denotes the domain of $g$ and $CD_g$ denotes the codomain of $g$. The image of $g$ is the set $\mathrm{Im}_g = \set{g(x)}{x \in D_g}$.
	\end{nt}

	If we say "lets consider the function $x \stackrel{h}{\mapsto} \sqrt{x^2-1}$", we mean that $h$ is the name of the function, $x \mapsto \sqrt{x^2-1}$ is its rule, $D_h = \ ]\infty, -1] \cup [1, \infty[$ and $CD_h = \R$.

	It is important to remark here that, in Calculus, we do not work with any kind of domain for our functions. We just work with functions whose domains are intervals or unions of intervals. Besides that, we are consider that an interval has infinity many elements, so the sets $]a,a[ \, = \varnothing$ and $[a,a] = \{a\}$ are not intervals to us.

	\begin{df}
		Let $A \subset \R$ be na interval or a union of intervals. In order to make things easier, lets suppose also that $A$ have to following properties:
		\begin{enumerate}
			\item If $p \in A$ is a left accumulation point of $A$, then $]p-\epsilon,p] \subset A$, for some $\epsilon > 0$.
			\item If $p \in A$ is a right accumulation point of $A$, then $[p,p+ \epsilon[ \ \subset A$, for some $\epsilon > 0$.
			\item If $p \in A$ is not a left accumulation point of $A$, then $]p-\epsilon,p] \not\subset A$, for all $\epsilon > 0$.
			\item If $p \in A$ is not a right accumulation point of $A$, then $[p,p+ \epsilon[ \ \not\subset A$, for all $\epsilon > 0$.
		\end{enumerate}
	\end{df}

	Now we will need a equivalence relation between differentiable functions.

	\begin{df}
		Let $\mathcal{F}$ be the set of all differentiable real functions, that is,
		\[\mathcal{F} = \set{f \colon A \to \R}{\text{$A$ is standard and $f$ is differentiable}}.\]
	
		If $f, g \in \mathcal{F}$, we will say that $f$ and $g$ are \textbf{equivalent} (or that $f$ is equivalent to $g$) if $f' = g'$. If $f$ and $g$ are equivalent, we will write $f \sim g$.
	
		If $f \in \mathcal{F}$, the equivalence class of $f$ will be denoted by $[f]$, that is,
		\[[f] = \set{g \in \mathcal{F}}{g \sim f}.\]
	\end{df}

	It is important to remark that, if $f$ and $g$ are differentiable, than $f'= g'$ means that $D_f = D_g$ and that $f'(x) = g'(x)$, for all $x \in D_f$.

	We need some operations between the equivalence classes of $\cal F$.

	\begin{df}
		Let $f, g \in \mathcal{F}$ be functions such that $D_f \cap D_g$ is standard and let $\alpha \in \R$. We will define $[f]+[g]$ and $\alpha [f]$
		\begin{align*}
			& [f]+[g] = \set{\varphi+\psi}{\varphi \in [f] \ \text{and} \ \psi \in [g]}; \\
			& [f]-[g] = \set{\varphi-\psi}{\varphi \in [f] \ \text{and} \ \psi \in [g]}; \\
			& \alpha [f] = \set{\alpha \varphi}{\varphi \in [f]}.
		\end{align*}
	\end{df}

	Lets remember here which are the functions $f+g$, $f-g$ and $\alpha f$.
	\begin{align*}
		& \begin{matrix}
			f+g \colon & D_f \cap D_g & \longrightarrow & \R \\
			& x & \longmapsto & f(x) + g(x)
		\end{matrix} \\
		& \begin{matrix}
			f-g \colon & D_f \cap D_g & \longrightarrow & \R \\
			& x & \longmapsto & f(x) - g(x)
		\end{matrix} \\
		& \begin{matrix}
			\alpha f \colon & D_f & \longrightarrow & \R \\
			& x & \longmapsto & \alpha \cdot f(x).
		\end{matrix}
	\end{align*}

	Its easy to show the following Lemma:

	\begin{lem}\label{operations}
		Let $f, g \in \mathcal{F}$.
		\begin{enumerate}
			\item If $\alpha \in \R$ and $\alpha \ne 0$, then $\alpha \cdot [f] = [\alpha \cdot f]$.
			\item If $D_f\cap D_g$ is standard, then $[f]+[g] = [f+g]$ and $[f-g] = [f]-[g]$.
		\end{enumerate}
	\end{lem}

	Now we can talk about primitives.

	\begin{df}
		Let $f$ and $F$ be two real functions with $F$ differentiable. We say that $F$ is a \textbf{primitive} of $f$ if $F' = f$. We will denote the set of all primitives of $f$ by $P(f)$:
		\[P(f) = \set{ g \colon D_f \to \R}{g'= f}.\]
	\end{df}

	Lets remark two things here:
	\begin{enumerate}
		\item When we say that $F' = f$, we mean that $D_{F'} = D_f$, $CD_{F'} = CD_f$ and that $F'(x) = f(x)$, for all $x \in D_f$. Usually $CD_F = CD_{F'} = CD_f = \R$ and $D_{F'} = D_F$, because $F$ is differentiable. So, if we want to check if $F$ is a primitive of $f$, we usually have to check if $D_F = D_f$ and if $F'(x) = f(x)$, for all $x \in D_f$.
		\item $F'=f \Leftrightarrow P(f) = [F]$.
	\end{enumerate}
	
	It follows from Lemma \ref{operations} the folowing Lemma:

	\begin{lem}\label{operations prim}
		Let $f$ and $g$ be functions such that $P(f) \ne \varnothing$ and $P(g) \ne \varnothing$.
		\begin{enumerate}
			\item If $\alpha \in \R$ and $\alpha \ne 0$, then $P(\alpha f) = \alpha P(f)$.
			\item If $D_f \cap D_g$ is standard, then $P(f+g) = P(f)+P(g)$ and $P(f-g) = P(f) - P(g)$.
		\end{enumerate}
	\end{lem}

\subsection{The Fundamental Theorems of Calculus and others}

	With our new definitions and notations, we can rewrite Theorem \ref{F1}:

	\begin{theo}[First Fundamental Theorem of Calculus]\label{t1}
		Let $I$ be an interval, $x_0 \in I$ and $f \colon I \to \R$ a function. If $f$ is \textbf{continuous}, then function $F \colon I \to \R$ given by $F(x) = \int_{x_0}^x f(t)dt$ is differentiable and $F \in P(f)$.
	\end{theo}

	Lets just remember here that $f$ is continuous in $[a,b]$, then $f$ is integrable in $[a,b]$. Thus, the function $F$ is well defined and Theorem \ref{t1} makes sense.
	
	If we just want to know if some function $f$ has a primitive, we can use Theorem \ref{t1}: if $D_f$ is an interval and $f$ is continuous, then $f$ has a primitive, that is, $P(f) \ne \varnothing$.
	
	We have to remark here that, if we want to make things easier to our students, we can limit ourselves to study primitives of functions whose domains are intervals, instead of studying primitives of functions whose domains are standard.

	If $D_f$ is standard but not an interval, then $D_f = \bigcup\limits_{i \in \mathcal{I}} I_i$, where $\cal I$ is a set of indexes, each $I_j$ is an interval, and $I_i \cap I_j = \varnothing$, when $i \ne j$. Considering $f$ continuous, we know that for each $i \in \mathcal{I}$, there exists a $F_i \colon I_i \to \R$ such that $F_i' = \left.f\right|_{I_i}$. Thus we can define the function $F \colon D_f \to \R$ by
	\[F(x) = F_i(x), \ \text{if} \ x \in I_i.\]
	Then it is easy to show that $F'=f$, that is, $P(f) \ne \varnothing$. Summarizing, we have the following Corollary:

	\begin{cor}\label{prim}
		If $D_f$ is standard and $f$ is continuous, then $P(f) \ne \varnothing$.
	\end{cor}

	Lets take a look on Integration by Parts:
	
	\begin{theo}[Integration by Parts]
		Let $D_f$, $D_g$ and $D_f \cap D_g$ be standard domains and let $f$ and $g$ be $C^1$ functions. With these assumptions, $P(fg') \ne \varnothing$, $P(f'g) \ne \varnothing$ and
		\[P(fg') = [fg] - P(f'g).\]
	\end{theo}
	
	\begin{proof}
		Lets first remark that, $D_f = D_{f'}$ and $D_g = D_{g'}$, because $f$ and $g$ are $C^1$ functions. Thus, $D_{fg} = D_{f'g} = D_{fg'} = D_f \cap D_g$ is standard.
		
		We know that $Fg \in P((fg)')$, so $P((fg)') \ne \varnothing$. We know also that $f$ and $g$ are $C^1$ functions, thus $fg'$ and $f'g$ are continuous, and it follows from Corollary \ref{prim} that $P(fg')\ne \varnothing$ and $P(f'g) \ne \varnothing$.
		
		Lets make some calculations.
		\[(fg)' = f'g + fg' \Rightarrow fg' = (fg)' - f'g.\]
		Thus, applying Lemma \ref{operations prim}, we have that
		\[P(fg') = P((fg)'-f'g) = P((fg)') - P(f'g) = [fg]-P(f'g). \qedhere\]
	\end{proof}

	One of the most important results proved in Calculus in order to develop techniques to find primitives is the following Theorem, which can be found in \cite{Spi}.

	\begin{theo}\label{f'=g'}
		Let $I$ be an interval and $f, g \colon I \to \R$. If $f$ and $g$ are differentiable and $f'=g'$, then there is a constant $c \in \R$ such that $f = g + c$.
	\end{theo}

	Based on this, it is easy to prove the following corollary:

	\begin{cor}\label{prim2}
		Let $I$ be an interval and $f \colon I \to \R$. If $F \in P(f)$ then $P(f) = \set{F+c}{c \in \R}$.
	\end{cor}

	The Corollary above is one of the main tools we use to calculate indefinite integrals. We will repeat here that, if we want to make Calculus easier, we can just study the primitives of functions whose domains are intervals, instead of functions whose domains are standard.

	Lets see an example before continuing.

	\begin{ex}
		Let $f \colon \R \to \R$ be the function given by $f(x) = x \cos(x)$. Then $f = \id \cdot \cos$, where $\id$ is the identity function on $\R$. Thus
		\begin{align*}
			& P(f) = P(\id\cdot \sin') = [\id\cdot \sin] - P(\id'\cdot \sin) = [\id\cdot \sin] - P(1 \cdot \sin) = \\
			& = [\id\cdot \sin]-[-\cos] = [\id\cdot \sin + \cos] = \\
			&= \set{x \mapsto x\sin(x)+ \cos(x) +c}{c \in \R}.
		\end{align*}
	\end{ex}
	
	The first Fundamental Theorem of Calculus uses the integral to define a primitive of a continuous function $f$. The Second Theorem of Calculus uses the primitive of a integrable function in order to calculate its integral:

	\begin{theo}[Second Fundamental Theorem of Calculus]\label{t2}
		Let $f \colon [a,b] \to \R$ a function. If $f$ is \textbf{integrable} and $F \in P(f)$, then $\int_a^b f(x)\dx = F(b) - F(a)$.
	\end{theo}

	\begin{cor}\label{t3}
		If $f \colon [a,b] \to \R$ is continuous, then there exists $F \colon [a,b] \to \R$ such that $F'=f$ and $\int_a^bf(x)\dx = F(b) - F(a)$.
	\end{cor}

\subsection{The integral $\int \frac{1}{x}\dx$}
	
	With our notation, we will find $P\left(x \mapsto \frac{1}{x}\right)$, instead of $\int \frac{1}{x} \, \dx$, which is the same.

	Let $f,F \colon \R^* \to \R$ be the functions given by the rules $f(x) = \frac{1}{x}$ and $F(x) = \ln|x|$. Its easy to see that $F \in P(f)$, then $P(f) = [F]$.
	
	But, $\left.F\right|_{\R_-^*} \in P\left(\left.f\right|_{\R_-^*}\right)$, and $\left.F\right|_{\R_+^*} \in P\left(\left.f\right|_{\R_+^*}\right)$. Thus, by Corollary \ref{prim2}, 
	\[P\left(\left.f\right|_{\R_-^*}\right) = \set{\left.F\right|_{\R_-^*} + c}{c \in \R} \quad \text{and} \quad P\left(\left.f\right|_{\R_+^*}\right) = \set{\left.F\right|_{\R_+^*} + c}{c \in \R}.\]

	Thus,
	\begin{align*}
		& \varphi \in P(f) \Leftrightarrow \left.\varphi\right|_{\R_-^*} \in P\left(\left.f\right|_{\R_-^*}\right) \ \text{and} \ \left.\varphi\right|_{\R_+^*} \in P\left(\left.f\right|_{\R_+^*}\right) \Leftrightarrow \\
		& \Leftrightarrow \left.\varphi\right|_{\R_-^*} = \left.F\right|_{\R_-^*} + c_1 \ \text{and} \ \left.\varphi\right|_{\R_+^*} = \left.F\right|_{\R_+^*} + c_2, \ \text{for some} \ c_1,c_2 \in \R \Leftrightarrow \\
		& \Leftrightarrow \varphi(x) =
		\begin{cases}
			\ln|x| + c_1, & \text{if} \ x < 0; \\
			\ln|x| + c_2, & \text{if} \ x > 0. \\
		\end{cases}
	\end{align*}

	Lets remark here that we could have used the following corollary which is very useful to work with functions whose domains are standard.
		
	\begin{cor}\label{prim3}
		Let $\cal I$ be a set of indexes and, for each $i \in \mathcal{I}$, let $I_i$ be an interval. Lets also suppose that $I_i \cap I_j = \varnothing$, if $i \ne j$. With the above hypothesis, if $D_f = \bigcup\limits_{i \in \mathcal{I}} I_i$ and $F \in P(f)$, then $\varphi \in P(f)$ if, and only if, for each $i \in \mathcal{I}$ there exists $c_i \in \R$ such that $\varphi(x) = F(x) + c_i$, for all $x \in I_i$.
	\end{cor}
	
\subsection{Changing variables without "mathemagic"}
	
	In Calculus, we have 3 situations we use changing of variables:
	\begin{enumerate}
		\item When we want to calculate $P((f\circ g)\cdot g')$ and we first calculate $P(f)$.
		\item When we want to calculate $P(f)$ but it is difficult and we first calculate $P((f\circ g)\cdot g')$, where $g$ is a convenient function we choose.
		\item When we want to calculate one of the sides of the equality $\int_{g(a)}^{g(b)} f(x) \dx = \int_a^b f(g(x))\cdot g'(x) \dx$, but we calculate the other side instead, because it is easier.
	\end{enumerate}

	We can use 4 theorems to solve this 3 situations. Lets see the theorems:

	\begin{theo}\label{x-u indefinida}
		Let $f$ and $g$ be two real functions such that $g$ is differentiable and $\mathrm{Im}_g \subset D_f$.
		\begin{enumerate}
			\item If $F \in P(f)$, then $P((f\circ g)\cdot g') = [F\circ g]$.
			\item If $F \in D_f$ and $D_g$ is an interval, then $P((f\circ g)\cdot g') = \set{F\circ g + c}{c \in \R}$.
		\end{enumerate}
	\end{theo}
		
	\begin{proof}
		Lets suppose that $F \in P(f)$. That means that $F$ is differentiable and that $F' = f$. Then $\mathrm{Im}_g \subset D_f = D_F$. Using the chain rule, we have that $(F\circ g)' = (F'\circ g)\cdot g' = (f \circ g)\cdot g'$. Therefore $F\circ g \in P((f\circ g)\cdot g')$.
			
		Now, supposing also that $D_g$ is an interval, then $D_{f\circ g} = D_g = D_{(f\circ g)\cdot g'}$ is an interval and, by the Corollary \ref{prim2}, $P((f\circ g)\cdot g') = \set{F\circ g + c}{c \in \R}$.
	\end{proof}

	\begin{theo}\label{x-u inversa}
		Let $g$ be a diffeomorphism.
		\begin{enumerate}
			\item If $\mathrm{Im}_g = D_f$, then $P(f) = \set{H \circ g^{-1}}{H \in P((f\circ g)\cdot g')}$.
			\item If $\mathrm{Im}_g = D_f$, $D_f$ is an interval and $H \in P((f\circ g)\cdot g')$ , then $P(f) = \set{H\circ g^{-1} + c}{c \in \R}$.
		\end{enumerate}
	\end{theo}
	
	\begin{proof}
		Let $F \in P(f)$. Then $F = F\circ g \circ g^{-1}$, because $D_F = D_f = \mathrm{Im}_g$. If we call $G = F \circ g$, then $F = G \circ g^{-1}$ and
		\[G' = (F\circ g)' = (F'\circ g)\cdot g'= (f\circ g)\cdot g'.\]
		Therefore $G \in P((f\circ g)\cdot g')$ and $F \in \set{H \circ g^{-1}}{H \in P((f\circ g)\cdot g')}$.
		
		Lets now suppose that $F \in \set{H \circ g^{-1}}{H \in P((f\circ g)\cdot g')}$. Then there exists an $G \in P((f\circ g)\cdot g')$ such that $F = G \circ g^{-1}$. Therefore $D_F = D_{g^{-1}} = D_f$ and
		\begin{align*}
			& F' = \left(G'\circ g^{-1}\right)\cdot {g^{-1}}' =  \left[ \left((f \circ g)\cdot g'\right) \circ g^{-1}\right] \cdot {g^{-1}}' = \\
			& = \left(f \circ g \circ g^{-1}\right)\cdot \left(g'\circ g^{-1}\right) \cdot {g^{-1}}' = f \cdot \left(g \circ g^{-1} \right) = f.
		\end{align*}
		Therefore $F \in P(f)$ and we conclude that
		\[P(f) = \set{H \circ g^{-1}}{H \in P((f\circ g)\cdot g')}.\]
	
		If $D_f$ is an interval and $H \in P\left((f\circ g)\cdot g'\right)$, then $H \circ g^{-1} \in P(f)$. Thus, by Corollary \ref{prim2}, $P(f) = \set{H\circ g^{-1} + c}{c \in \R}$.
	\end{proof}

	In the theorems above, we do not need to suppose that $f$ is continuous neither that $g'$ is continuous, because we are not using the First Fundamental Theorem of Calculus.

	Sometimes the theorems above are not sufficient, so we have another one:

	\begin{theo}\label{x-u not diffeomorfismo}
		Let $g\colon D_g \to D_f$ be a differentiable and bijective function. Lets suppose also that $g^{-1}$ is continuous and that $g^{-1}$ is differentiable in the interior of its domain. 
		\begin{enumerate}
			\item If $f \colon D_f \to \R$ is continuous, then $P(f) = \set{H\circ g^{-1}}{H \in P((f\circ g)\cdot g')}$.
			\item If $D_f$ is an interval and $H \in P((f\circ g)\cdot g')$, then $P(f) = \set{H \circ g^{-1} + c}{c \in \R}$.
		\end{enumerate}
	\end{theo}
	
	\begin{proof}[Proof of Theorem \ref{x-u not diffeomorfismo}]
		Just like it was done at Theorem \ref{x-u inversa}, we can prove that $P(f) \subset \set{H\circ g^{-1}}{H \in P((f\circ g)\cdot g')}$.
		
		Lets suppose now that $F \in \set{H\circ g^{-1}}{H \in P((f\circ g)\cdot g')}$. Then there exists an $G \in P((f\circ g)\cdot g')$ such that $F = G \circ g^{-1}$. Therefore $D_F = D_{g^{-1}} = D_f$.
		
		Let $p \in D_F$ be an interior point. Thus
		\begin{align*}
			& F'(p) = G'\left(g^{-1}(p)\right)\cdot {g^{-1}}'(p) =  \left[ \left(\left(f \circ g\right) \left(g^{-1}(p)\right)\cdot g'\left(g^{-1}(p) \right)\right) \right] \cdot {g^{-1}}'(p) = \\
			& = f(p)\cdot g'\left(g^{-1}(p)\right) \cdot {g^{-1}}'(p) = f(p) \cdot \left(g \circ g^{-1} \right)'(p) = f(p) \cdot 1 = f(p).
		\end{align*}
		
		Now, lets suppose that $p \in D_f$ is not an interior point, that is, $p$ is an endpoint of one of the intervals which form $D_f$ (once $D_f$ is an interval or an union of intervals). Without loss of generality, we can suppose that $p$ is a left end point of one of the intervals that form $D_f$, but $p$ is not a left accumulation point of $D_f$. Thus,
		\[\lim_{x \to p} \frac{F(x)-F(p)}{x-p} = \lim_{x \to p^+} \frac{\cancelto{0}{F(x)-F(p)}}{\cancelto{0}{x-p}} = \lim_{x \to p^+} \frac{F'(x)}{1} = \lim_{x\to p^+} f(x) = f(p).\]
		In the above calculations, we have used the fact that $F'(x) = f(x)$ if $x$ is an interior point of $D_f$, and we have also used the L'Hospital rules and the hypothesis that $f$ is continuous.
		
		Therefore $F'(p) = f(p)$. Thus, $F'(x) = f(x)$, for all $x \in D_F = D_f$. That is $F \in P(f)$.
		
		The second part of the Theorem follows form Corollary \ref{prim2}.
	\end{proof}

	\begin{theo}\label{x-u definida}
		Let $g \colon [a,b] \to \R$ a $C^1$ function and $f \colon [c,d] \to \R$ a continuous function such that $g([a,b]) \subset [c,d]$. Then
		\[\int_{g(a)}^{g(b)} f(x)\dx = \int_a^b f(g(x))\cdot g'(x) \dx.\]
	\end{theo}
	
	\begin{proof}
		By the First Fundamental Theorem of Calculus \ref{t1}, $P(f) \ne \varnothing$, because $f$ is continuous.
		
		Let $F \in P(f)$. Then $F\circ g \in P((f\circ g)\cdot g')$ and $(f\circ g)\cdot g'$ is continuous. Using the Corollary \eqref{t3}, we have that
		\[\int_{g(a)}^{g(b)} f(x) \dx = F(g(b)) - F(g(a)) = \int_a^b f(g(x))\cdot g'(x)\dx. \qedhere\]
	\end{proof}
	
	We want to remark that we just can apply the Theorem \ref{x-u definida} if $f$ and $g'$ are continuous, because we are using the First and Second Fundamental Theorems of Calculus for the functions $f$ and $(f\circ g)\cdot g'$. If we do not know if $f$ is continuous or $g'$ is continuous, then we have to assume that $P(f) \ne \varnothing$, that $f$ is integrable and that $(f \circ g)\cdot g'$ is also integrable.
	
	Now we can solve the examples of Section \ref{magic} without "mathemagic".
	
	\begin{ex}[$\int\cos \left(x^2\right) x \dx$]
		First of all, lets give names to the functions:
		\[\begin{matrix}
			f : & \R & \longrightarrow & \R \\
			& x & \longmapsto & \cos(x^2)x
		\end{matrix} \qquad \text{and} \qquad
		\begin{matrix}
			g : & \R & \longrightarrow & \R \\
			& x & \longmapsto & x^2
		\end{matrix}\]
	
		We want to calculate $P(f)$, but $f(x) = \cos(x^2)x = \cos(g(x))\cdot \frac{g'(x)}{2}$, that is, $f = \frac{1}{2}\cdot (\cos \circ g)\cdot g'$. Thus
		\begin{align*}
			& P(f) = P\left(\frac{1}{2}\cdot (\cos \circ g)\cdot g'\right) = \frac{1}{2} P\left((\sin \circ g)'\right) = \\
			& = \frac{1}{2} \set{\sin \circ g + c}{c \in \R} = \set{\frac{1}{2} \sin \circ g + c}{c \in \R}.
		\end{align*}
	\end{ex}
	
	\begin{ex}[$\int \frac{1}{\sqrt{x^2 +1}} \dx$]
	
		Let be $f \colon \R \to \R$ given by $f(x) = \frac{1}{\sqrt{x^2+1}}$. In order to calculate $P(f)$, we just need to find $H \in P((f\circ g)\cdot g')$, where $g$ is a convenient diffeomorphism such that $\mathrm{Im}_g = D_f$. Then, by Theorem \ref{x-u inversa}, $P(f) = \set{H \circ g^{-1} + c}{c \in \R}$.
	
		Lets use $g = \left.\tan\right|_I$, where $I = \left]-\frac{\pi}{2}, \frac{\pi}{2} \right[$. We know $g$ is a diffeomorphism..
		
		By another side,
		\begin{align*}
			& \left((f\circ g)\cdot g'\right)(x) = f(\tan(x))\cdot \tan'(x) = \\
			&= \frac{1}{\sqrt{\tan^2(x) + 1}}\cdot \sec^2(x) = \cos(x)\cdot\sec(x) = \sec(x).
		\end{align*}
		Therefore $P((f \circ g)\cdot g') = P(\sec|_I)$, and, if we already know $P\left(\sec|_I\right)$, we can write
		\begin{align*}
			&\varphi \in  P((f\circ g)\cdot g') = P\left(\sec|_I\right) \Leftrightarrow \\
			& \Leftrightarrow \varphi(x) = \ln|\tan x + \sec x| + c, \forall x \in I,
		\end{align*}
		where $c \in \R$ is constant.
		
		Taking $H \colon I \to \R$ given by $H(x) = \ln |\tan(x) + \sec(x)|$, and applying Theorem \ref{x-u inversa}, we have that
		\[P(f) = \set{H \circ \tan^{-1} + c}{c \in \R}.\]
		But,
		\[H\left(\tan^{-1}(x)\right) = \ln\left|x + \sec\left(\tan^{-1}(x)\right)\right| = \ln \left| x + \sqrt{x^2+1}\right|.\]
		Therefore
		\[h \in P(f) \Leftrightarrow h(x) = \ln\left(x+\sqrt{x^2+1}\right) + c, \ \forall x \in \R, \ \text{for some constant} \ c \in \R.\]
	\end{ex}

	\begin{ex}[$\int \sqrt{1-x^2} \dx$]
		Let $f \colon [-1,1] \to \R$ be given by $f(x) = \sqrt{1-x^2}$. The function $f$ is obviously continuous.
		
		The figure of example \ref{ex nao difeo}, in the last section, gives us the idea of using the function $g(\theta) = \sin(\theta)$. So, let $g \colon \left[-\frac{\pi}{2}, \frac{\pi}{2}\right] \to [-1,1]$ given by $g(\theta) = \sin(\theta)$. Thus $g$ is differentiable and bijective, and $g^{-1} \colon [-1,1] \to \left[-\frac{\pi}{2}, \frac{\pi}{2}\right]$ is continuous and differentiable in $]-1,1[$. Therefore we can apply Theorem \ref{x-u not diffeomorfismo} to conclude that $P(f) = \set{H\circ g^{-1} + c}{c \in \R}$, where $H \in P((f\circ g)\cdot g')$.
		
		But,
		\[f(g(\theta))\cdot g'(\theta) = \sqrt{1-\sin^2(\theta)}\cdot \cos(\theta) = \cos^2(\theta), \ \forall \theta \in \left[-\frac{\pi}{2}, \frac{\pi}{2}\right].\]
		
		Thus, $P((f\circ g)\cdot g') = P\left({\cos^2}|_I \right)$, where $I = \left[ -\frac{\pi}{2}, \frac{\pi}{2} \right]$. By another side, we know that the function $H \colon I \to \R$, given by $H(\theta) = \frac{\theta}{2} + \frac{\sin(\theta)\cos(\theta)}{2}$, is a primitive of ${\cos^2}_I = (f\circ g)\cdot g'$. Therefore, $P(f) = \set{H\circ g^{-1} + c}{c \in \R}$.
		
		But
		\begin{multline*}
			H\left(g^{-1}(x) \right) = \frac{\arcsin(x)}{2} + \frac{\sin(\arcsin(x))\cos(\arcsin(x))}{2} = \\
			= \frac{\arcsin(x)}{2} + \frac{x\sqrt{1-x^2}}{2}.
		\end{multline*}
		Thus $P(f) = \set{x\mapsto \frac{\arcsin(x)}{2} + \frac{x\sqrt{1-x^2}}{2} + c}{c \in \R}$.
	\end{ex}

	\begin{ex}[$\int_0^{\frac{3\pi}{4}} e^{\cos(x)}\cdot \sin(x) dx$]
		Lets consider the exponential function $\exp \colon \R \to \R$, given by $\exp(x) = e^x$. We know that the functions $\exp$, $\cos$ and $\cos' = -\sin$ are continuous and that $\cos\left(\left[0,\frac{3\pi}{4}\right]\right) \subset \R = D_{\exp}$, thus, applying the theorem \ref{x-u definida}, we have that
		\begin{align*}
			& \int_0^{\frac{3\pi}{4}} e^{\cos(x)}\cdot \sin(x) \dx = -\int_0^{\frac{3\pi}{4}} e^{\cos(x)}\cdot \cos'(x) \dx = -\int_{\cos(0)}^{\cos\left(\frac{3\pi}{2}\right)} e^x \dx = \\
			& = - \int_1^{-\frac{\sqrt 2}{2}} e^x \dx = \int_{\frac{\sqrt 2}{2}}^1 e^x \dx = e-\frac{1}{e^{\frac{\sqrt 2}{2}}}.
		\end{align*}
	\end{ex}

	
		
		

\subsection{The integral $\int\frac{1}{1-\cos x + \sin x}\dx$}

	When we just say the function $\frac{1}{1-\cos x + \sin x}$, we mean that the function's rule is $x \mapsto \frac{1}{1-\cos x + \sin x}$ and the function's domain is the biggest subset of $\R$ where the rule makes sense. We cannot divide by 0, so the functions domain is the set
	\[\set{x \in \R}{1-\cos x + \sin x \ne 0}.\]
	
	Lets call $f(x) = \frac{1}{1-\cos x + \sin x}$ and calculate $D_f$. We believe that a university level student should take the trouble to calculate the domain of $f$ before anything else. But here again, we have to remark that it would be much easier to just calculate primitives of functions whose domains are intervals.
	
	Lets consider the function $\phi(x) = 1 - \cos x + \sin x$, with $x \in \R$. Then $D_f = \set{x \in \R}{\phi(x) \ne 0}$ and $\phi'(x) = \sin x + \cos x$. Thus
	\begin{enumerate}
		\item $\phi'(x) = 0 \iff x \in \set{\frac{3 \pi}{4} +k\pi}{k \in \Z}$;
		\item $\phi'(x) > 0 \iff x \in \, \bigcup\limits_{k \in \Z}\left]-\frac{\pi}{4} + 2k\pi, \frac{3\pi}{4} + 2k\pi \right[$;
		\item $\phi'(x) < 0 \iff x \in \, \bigcup\limits_{k \in \Z}\left] \frac{3\pi}{4} + 2k\pi, 2\pi - \frac{\pi}{4} + 2k\pi\right[$.
	\end{enumerate}
	
	The following figure shows in the trigonometric circle the points where $\phi'(x)=0$, the points where $\phi'(x) > 0$ and the points where $\phi'(x) < 0$.
	
	\begin{center}
		\psfrag{a}{$\phi'(x)= 0$}
		\psfrag{b}{$\phi'(x)> 0$}
		\psfrag{c}{$\phi'(x)= 0$}
		\psfrag{d}{$\phi'(x)< 0$}
		\includegraphics[scale=0.6]{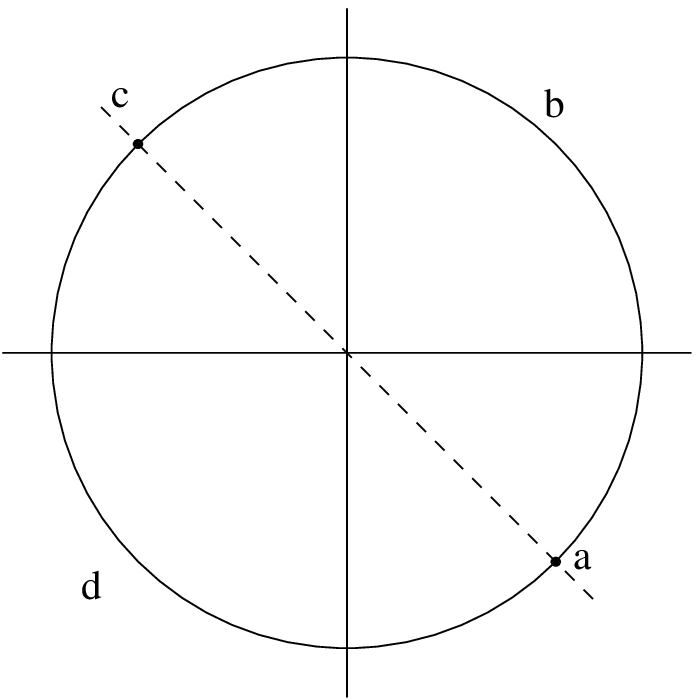}
	\end{center}
	
	This means that $\phi$ is strictly increasing in the intervals $\left[-\frac{\pi}{4} + 2k\pi, \frac{3\pi}{4} + 2k\pi\right]$, for all $k \in \Z$, and $\phi$ is strictly decreasing in the intervals $\left[ \frac{3\pi}{4} + 2k\pi, 2\pi - \frac{\pi}{4} + 2k\pi\right]$, for all $k \in \Z$. Thus, $\phi$ has no more than one zero in each of the following intervals:
	\[\left[ 2k\pi -\tfrac{\pi}{4}, 2k\pi + \tfrac{3\pi}{4} \right] \quad \text{and} \quad \left[ 2k\pi + \tfrac{3\pi}{4}, (2k+2)\pi - \tfrac{\pi}{4} \right], \ \forall k \in \Z.\]
		
	But,
	\begin{enumerate}
		\item $2k\pi \in \left[ 2k\pi - \frac{\pi}{4}, 2k\pi + \frac{3\pi}{4}\right]$,
		\item $2k\pi + \frac{3\pi}{2} \in \left[ \frac{3\pi}{4} + 2k\pi, (2k+2)\pi - \frac{\pi}{4} \right]$ and
		\item $\phi(2k\pi) = \phi\left(2k\pi + \frac{3\pi}{2}\right) = 0$.
	\end{enumerate}
	Therefore,
	\[\set{x\in \R}{\phi(x) = 0} = \set{2k\pi, 2k\pi + \tfrac{3\pi}{2}}{k \in \Z}.\]
	We conclude that
	\begin{multline*}
		D_f = \set{x \in \R}{x \ne 2k\pi \ \text{and} \ x \ne 2k\pi + \tfrac{3\pi}{2}, \ \forall k \in \Z} = \\
		= \bigcup_{k \in \Z} \left( \, \left]2k\pi, 2k\pi + \tfrac{3\pi}{2} \right[ \cup \left] 2k\pi + \tfrac{3\pi}{2}, 2(k+2)\pi \right[ \, \right).
	\end{multline*}
		
	Now, for every $k \in \Z$, let
	\[I_1(k) = \left]2k\pi, 2k\pi+\tfrac{3\pi}{2} \right[ \quad \text{and} \quad I_2(k) = \left] 2k\pi+\tfrac{3\pi}{2}, (2k+2)\pi \right[.\]
	Thus, $D_f = \bigcup\limits_{k \in \Z} \left(I_1(k) \cup I_2(k) \right)$.

	Now we know that $D_f$ is an union of infinite disjoint open intervals. In order to calculate $P(f)$, we will first calculate $P\left(\left.f\right|_I\right)$, where $I$ is an interval.

	Using the identities
	\[\cos x = \frac{1-\tan^2\frac{x}{2}}{1+\tan^2\frac{x}{2}} \quad \text{and} \quad \sin x = \frac{2\tan\frac{x}{2}}{1+\tan^2\frac{x}{2}},\]
	we can write
	\begin{equation}\label{trig}
		\frac{1}{1-\cos x + \sin x} = \frac{1+\tan^2\frac{x}{2}}{2\tan^2\frac{x}{2}+2\tan\frac{x}{2}}
	\end{equation}
	
	Lets remark here that the identities above do not make sense for some points. For example, when $\frac{x}{2} = \frac{\pi}{2} + k\pi$ and $k \in \Z$, the identities do not make sense, because the tangent function is not defined at those points. So, lets suppose that both sides of equation \eqref{trig} make sense for every $x \in I$.
	
	Let $g \colon I \to J$ be given by $g(x) = \tan\frac{x}{2}$, where $J = g(I)$. Then $g'(x) = \frac{\tan'(x)}{2} = \frac{1}{2}\left(1 + \tan^2 \frac{x}{2}\right)$. Then equation \eqref{trig} becomes
	\[\frac{1}{1-\cos x + \sin x} = \frac{2g'(x)}{2\left(g^2(x)+g(x) \right)} = h(g(x))\cdot g'(x),\]
	where $h \colon J \to \R$ is given by $h(x) = \frac{1}{x^2+x}$.

	Using Theorem \ref{x-u indefinida}, $P\left(\left.f\right|_I\right) = P((h\circ g)\cdot g') = \set{H \circ g + c}{c \in \R}$, where $H \in P(h)$. But $h(x) = \frac{1}{x^2+x} = \frac{1}{x} - \frac{1}{x+1}$. If $h_1, h_2 \colon J \to \R$ are given by $h_1(x) = \frac{1}{x}$ and $h_2(x) =\frac{1}{x+1}$, then $h = h_1 - h_2$ and
	\[P(h) = P(h_1-h_2) = P(h_1) - P(h_2).\]
	By another side
	\begin{align*}
		& P(h_1) = \set{\varphi \colon J \to \R}{\varphi(x) = \ln|x|+ c, \ \forall x \in I, \ \text{where $c$ is constant}}; \\
		& P(h_2) = \set{\varphi \colon J \to \R}{\varphi(x) = \ln|x+1|+ c, \ \forall x \in I, \ \text{where $c$ is constant}}.
	\end{align*}
	Therefore the function $H \colon J \to \R$ given by $H(x) = \ln|x|-\ln|x+1| = \ln\left|\frac{x}{x+1}\right|$ is a primitive of $h$ and
	\[P(\left.f\right|_I) = \set{H \circ g + c}{c \in \R}.\]
	
	Until here, we have that $\varphi \in P\left(\left.f\right|_I\right)$ if, and only if, $\varphi \colon I \to \R$ is given by
	\[\varphi(x) = H(g(x)) + c = \ln\left|\frac{\tan\frac{x}{2}}{1+\tan\frac{x}{2}} \right| + c,\]
	where $c$ is constant.

	By the same calculations of the Subsection \ref{last}, we know that
	\[\ln\left| \frac{\tan\frac{x}{2}}{1+\tan\frac{x}{2}}\right| = \ln \left|\frac{\sin x}{1 + \cos x + \sin x}\right|,\]
	every time when both sides of the equality above makes sense.
	
	Lets now consider the domain of the function $x \stackrel{f_2}{\mapsto} \ln \left|\frac{\sin x}{1 + \cos x + \sin x}\right|$, which is the biggest subset of $\R$ in which the rule makes sense. For this, lets consider the auxiliary function $\psi \colon \R \to \R$ given by $\psi(x) = 1 + \cos x + \sin x$. Its clear that if $\psi(x) = 0$, then $x \notin D_\psi$. So lets find the zeros of $\psi$.
	
	We know that $\psi'(x) = \cos x - \sin x$. Thus
	\begin{enumerate}
		\item $\psi'(x) = 0 \iff x \in \set{\frac{\pi}{4} + k\pi}{k \in \Z}$,
		\item $\psi'(x) < 0 \iff x \in \left] 2k\pi + \frac{\pi}{4}, (2k+1)\pi + \frac{\pi}{4}\right[$,
		\item $\psi'(x) > 0 \iff x \in \left] (2k+1)\pi + \frac{\pi}{4}, (2k+2)\pi + \frac{\pi}{4}\right[$.
	\end{enumerate}

	The next figure shows in the trigonometric circle the points where $\psi'(x) =0$, the points where $\psi'(x) >0$ and the points where $\psi'(x)<0$.

	\begin{center}
		\psfrag{a}{$\psi'(x) = 0$}
		\psfrag{b}{$\psi'(x) < 0$}
		\psfrag{c}{$\psi'(x) = 0$}
		\psfrag{d}{$\psi'(x) > 0$}
		\includegraphics[scale=0.6]{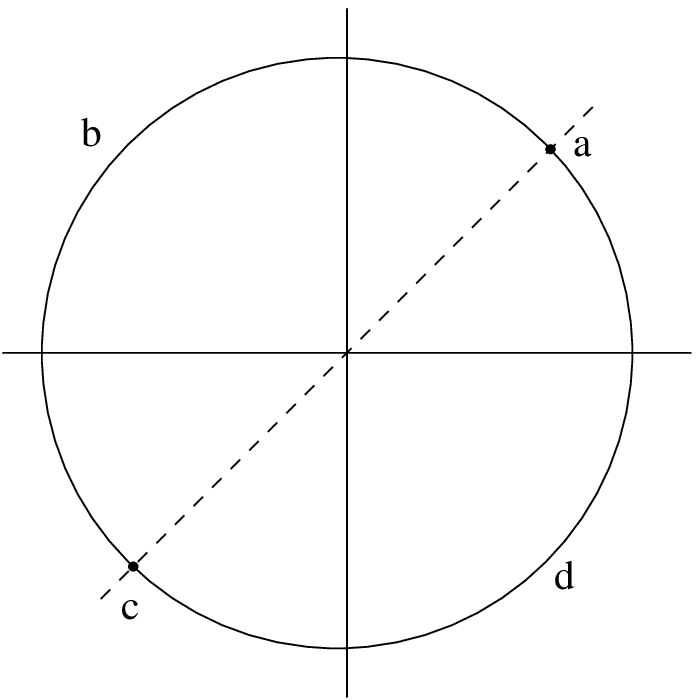}
	\end{center}
	
	This means that $\psi$ is strictly increasing in the intervals
	\[\left[ (2k+1)\pi + \frac{\pi}{4}, (2k+2)\pi + \frac{\pi}{4} \right], \ \forall k \in \Z,\]
	and $\psi$ is strictly decreasing in the intervals
	\[\left[ 2k\pi + \frac{\pi}{4}, (2k+1)\pi + \frac{\pi}{4}\right],  \forall k \in \Z.\]
	Therefore $\psi$ has no more than one zero in each of the following intervals:
	\[\left[ 2k\pi +\tfrac{\pi}{4}, (2k+1)\pi + \tfrac{\pi}{4} \right] \quad \text{and} \quad \left[ (2k+1)\pi + \tfrac{\pi}{4}, (2k+2)\pi + \tfrac{\pi}{4} \right].\]
		
	By another side,
	\begin{enumerate}
		\item $(2k+1)\pi \in \left[ 2k\pi +\frac{\pi}{4}, (2k+1)\pi + \frac{\pi}{4} \right]$,
		\item $2k\pi + \frac{3\pi}{2} \in \left[ (2k+1)\pi + \frac{\pi}{4}, (2k+2)\pi + \frac{\pi}{4} \right]$,
		\item $\psi((2k+1)\pi) = \psi\left(2k\pi + \frac{3\pi}{2}\right) = 0$, and
		\item $\sin(k\pi) = 0$.
	\end{enumerate}
	Therefore, $\set{x\in \R}{\psi(x) = 0} = \set{(2k+1)\pi, 2k\pi + \frac{3\pi}{2}}{k \in \Z}$ and we conclude that
	\[D_{f_2} = \bigcup_{k \in \Z} \left( \left]2k\pi, (2k+1)\pi \right[ \cup \left] (2k+1)\pi, 2k\pi + \tfrac{3\pi}{2} \right[ \cup \left] 2k\pi+\tfrac{3\pi}{2}, (2k+2)\pi\right[ \right).\]
	
	Here we observe that $D_{f_2} \subset D_f$ and that $D_f \setminus D_{f_2} = \set{(2k+1)\pi}{k \in \Z}$.
	
	Lets calculate the following limit:
	\begin{align*}
		& \lim_{x \to (2k+1)\pi} f_2(x) = \lim_{x \to 2k\pi + \pi} \ln \left| \frac{\sin x}{1 + \cos x + \sin x}\right| = \\
		& = \lim_{x \to 2k\pi + \pi} \ln \left| \frac{\cos x}{\cos x - \sin x}\right| = \ln\left|\frac{-1}{-1}\right| = 0.
	\end{align*}

	Lets define $F \colon \bigcup\limits_{k \in \Z} \left( I_1(k) \cup I_2(k)\right)\to \R$ by
	\[F(x) = \begin{cases}
		\ln \left|\frac{\sin x}{1 + \cos x + \sin x}\right|, & \text{if} \ x \ne (2k+1)\pi; \ \\
		0, & \text{if} \ x = (2k+1)\pi;
	\end{cases}\]
	
	Now, $D_F = D_f$ and its possible to prove that $F' = f$, that is $F \in P(f)$. Thus, applying Corollary \ref{prim3}, we have that $\varphi \in P(f)$ if, and only if, for each $k \in \Z$, there exists $c_1(k), c_2(k) \in \R$ such that
	\[\varphi(x) = \begin{cases}
		\ln \left|\frac{\sin x}{1 + \cos x + \sin x}\right| + c_1(k), & \text{if $x \in I_1(k)$ and $x \ne (2k+1)\pi$}; \\
		c_1(k), & \text{if $x = (2k+1)\pi$}; \\
		\ln \left|\frac{\sin x}{1 + \cos x + \sin x}\right| + c_2(k), & \text{if $x \in I_2(k)$}.
	\end{cases}\]

\bibliographystyle{acm}
\addcontentsline{toc}{section}{References}
\bibliography{calculus.bib}
\end{document}